\documentclass{article}
\usepackage{a4wide, amsmath, amsthm, amssymb, amscd, mathptmx}

\usepackage{amsfonts}

\newcommand\imCMsym[4][\mathord]{%
	\DeclareFontFamily{U} {#2}{}
	\DeclareFontShape{U}{#2}{m}{n}{
		<-6> #25
		<6-7> #26
		<7-8> #27
		<8-9> #28
		<9-10> #29
		<10-12> #210
		<12-> #212}{}
	\DeclareSymbolFont{CM#2} {U} {#2}{m}{n}
	\DeclareMathSymbol{#4}{#1}{CM#2}{#3}
}

\imCMsym{cmmi}{124}{\Cjmath}

\usepackage{color}
\usepackage{pstricks}

\usepackage[all]{xy}\CompileMatrices\SelectTips{cm}{12}

\theoremstyle{plain}
\newtheorem{Thm}{\sc Theorem}[section]
\newtheorem{Theorem}[Thm]{\sc Theorem}
\newtheorem{Corollary}[Thm]{\sc Corollary}

\newtheorem*{Corollary*}{\sc Corollary}

\newtheorem{Proposition}[Thm]{\sc Proposition}
\newtheorem*{Proposition*}{\sc Proposition}
\newtheorem{Lemma}[Thm]{\sc Lemma}

\theoremstyle{definition}

\theoremstyle{remark}
\newtheorem{Remark}[Thm]{Remark}

\newtheorem*{Example*}{Example}
\newtheorem*{Remark*}{Remark}

\newcommand{\CC}{{\mathbb C}}

\newcommand{\ZZ}{{\mathbb Z}}

\newcommand{\QQ}{{\mathbb Q}}
\newcommand{\RR}{{\mathbb R}}

\newcommand{\cF}{{\mathcal F}}

\newcommand{\cO}{{\mathcal O}}

\newcommand{\cT}{{\mathcal T}}

\newcommand{\Pic}{{\mathop{\rm Pic \, }}}

\newcommand{\Coh}{{\mathop{\operatorname{Coh}\, }}}

\newcommand{\ch}{{\mathop{\rm ch \, }}}

\newcommand{\im}{\mathop{\rm im \, }}

\newcommand{\Refl}{{\mathop{\rm Ref\,}}}

\newcommand{\td}{{\mathop{\rm td \, }}}
\newcommand{\Td}{{\mathop{\rm Td \, }}}

\newcommand{\Spec}{{\mathop{{\rm Spec\, }}}}

\newcommand{\Sing}{{\mathop{{\rm Sing \,}}}}

\begin{document}

\markboth {\rm }{}

\title{Bridgeland stability conditions on normal surfaces}
\author{Adrian Langer} \date{\today}

\maketitle


{\footnotesize
{\noindent \sc Address:}\\
Institute of Mathematics, University of Warsaw,
ul.\ Banacha 2, 02-097 Warszawa, Poland\\
e-mail: {\tt alan@mimuw.edu.pl}
}

\medskip

\begin{abstract}
	We prove a new version of Bogomolov's inequality on normal proper surfaces. This allows to construct Bridgeland's stability condition on such surfaces. In particular, this gives the first known examples of stability conditions on non-projective, proper schemes. 
\end{abstract}

\medskip


\section*{Introduction}

Let  $X$ be a separated scheme locally of finite type over a fixed base field.
Let $K_0(X)$ be the Grothendieck group of coherent $\cO_X$-modules and let
$A_*(X)$ be the Chow group of $X$
(i.e., the quotient of the  group of cycles modulo rational equivalence).  
Then  there exists a homomorphism
$$\tau_X: K_0(X)\to A_{*}(X)_{\QQ},$$
which in the case of smooth $X$ coincides with the homomorphism $\alpha\to \ch(\alpha)\cap \Td (X)$
(see \cite[Theorem 18.3]{Fu} for more details). However, in general one cannot define any  (additive) homomorphism $\ch : K_0(X)\to A^{*}(X)_{\QQ}$ to the (rational) operational Chow ring $A^{*}(X)_{\QQ}=A^*(X\to X)\otimes \QQ$ so that $\tau _X (\alpha)= \ch(\alpha)\cap \Td (X)$
for all $\alpha \in K_0(X)$.
Such ``Chern character'' does not exist even on normal surfaces. But in this special case one can use Mumford's intersection theory of Weil divisors on a normal surface $S$ to define a substitute
that corresponds to the evaluation $\ch \cap [S]$. More precisely, we define \emph{Mumford's Chern character} 
$$ \ch^M=(\ch_0^M, \ch_1^M, \ch_2^M): K_0(S)\to A_*(S)_{\QQ}$$
so that for any $\alpha \in K_0(S)$ it satisfies the Riemann--Roch formula
$$\tau _S(\alpha)=[S]+ \left( \ch_1^M(\alpha ) -\frac{1}{2}K_{S} \right) +
\ch_2^M(\alpha)-\frac{1}{2}\ch_1^M(\alpha)\cdot   K_{S}+\ch _0^M (\alpha)\, \Td _2 ({S}),$$
where  $\ch_1^M(\alpha)\cdot   K_{S}$ denotes Mumford's intersection product of classes of Weil divisors.
It is easy to define $\ch_0^M$ and $\ch_1^M$. Then one can use the above formula to define the missing piece $\ch_2^M$. We study properties of this Chern character and in particular, we show how to compute it using a resolution of singularities of $S$.  

Using $\ch ^M$ we define \emph{Mumford's discriminant} of a coherent $\cO_S$-module $E$ by the formula
$$ \Delta^M(E):=\ch_1^M(E)^2-2\ch_0^M(E) \cdot \ch_2^M(E) .$$
This discriminant satisfies the following Bogomolov type inequality:
 
\begin{Theorem}\label{main}
Let $S$ be a normal proper surface over an algebraically closed field.
Then there exists an explicit constant $C_{S}$, depending on the singularities  and the birational equivalence class of $S$, satisfying the following condition. For any numerically non-trivial nef Weil divisor $H$ on $S$ and any slope $H$-semistable torsion free coherent $\cO_S$-module $E$ we have
$$\int_S \Delta^M(E) +C_{S}\,\ch_0 ^M(E)^2\ge 0.$$
Moreover, if $S$ has only rational Gorenstein singularities then $C_S$ does not depend on the singularities of $S$.
\end{Theorem} 

The proof of Theorem \ref{main} depends on \cite{La-MZ} and \cite{La-Inters} that allow us to reduce to the smooth case. If $S$ is a smooth projective surface and  the base field has characteristic zero, the above theorem is classical. If $S$ is smooth, projective and  the base field has positive   characteristic, the above theorem was proven by N. Koseki in \cite{Ko} building on the methods of \cite{La-AnnMath}.  In case $X$ is a normal projective surface with only rational double points, Theorem \ref{main} was claimed in \cite{NS}. Unfortunately, the notion of a Chern character in \cite{NS} is not correct and the proofs contain various mistakes.

In the above theorem we use only the fact that we can define the number $\int_S \Delta^M(E)$.  If $E$ is a (non-zero) torsion free coherent $\cO_S$-module of rank  $r>0$ and $D$ is a Weil divisor such that $\det E\simeq \cO_S(D)$ then  this  number can be easily recovered from the Riemann--Roch formula
$$\chi (S, E)=-\frac{1}{2r}\int_S \Delta^M(E) + \frac{1}{2r}D. (D-rK_S)+r \chi (S, \cO_S).$$ 
In Theorem \ref{main} this formula can be taken as a defining property of $\int_S \Delta^M(E)$.

\medskip

Similarly to the smooth projective surface case, the above inequality allows to construct Bridgeland stability conditions on normal proper surfaces:

\begin{Theorem}\label{main2}
Let $S$ be a normal proper surface. Then there exist geometric Bridgeland stability conditions on $D^b(S)$ satisfying the full support property.
\end{Theorem}

In the above theorem $D^b(S)$ denotes the bounded derived category of coherent sheaves on $S$.
This is well-known to be equivalent to the bounded derived category of complexes with coherent cohomology. Theorem \ref{main2} follows immediately from  Lemma \ref{existence-num-ample} and Theorem \ref{explicit-stability-on-proper}.
 
 \medskip
 
It is well-known that every smooth proper surface is projective. In fact, a normal surface with only $\QQ$-factorial singularities is already projective. By Lipman's result (see \cite[Proposition 17.1]{Li}) a rational surface singularity is $\QQ$-factorial. So Theorem \ref{main2} for normal surfaces with at most rational singularities is still too weak to provide 
Bridgeland stability conditions on non-projective varieties.
 However, there exist many normal proper surfaces that are not projective, or even such that do not contain any non-trivial line bundles (see \cite{Sch}). Using such surfaces and Theorem \ref{main2}, we get the first examples of Bridgeland stability conditions on some proper, non-projective scheme.

\medskip

Note that Theorems \ref{main} and \ref{main2} are new also in the characteristic zero case.

\medskip

The structure of the paper is as follows. In Section 1 we gather a few preliminary results. Section 2 contains a study of  Mumford's Chern character. In Section 3 we prove Theorem \ref{main}.
The last section contains proof of Theorem \ref{main2}.

\medskip

\subsection*{Notation.}

In the paper we fix a base field $k$ (which unless otherwise stated need not be algebraically closed nor perfect). A variety  is an  irreducible, reduced scheme $X$ over $k$ such that the structure morphism $X\to \Spec k$ is separated and of finite type. It is called  proper if  $X\to \Spec k$ is proper.
A vector bundle on a variety $X$ is a locally free coherent $\cO_X$-module.
A surface is a $2$-dimensional variety. 

\section{Preliminaries}

\subsection{Intersection theory on normal surfaces}

Let $S$ be a normal surface. Let $f: \tilde S \to S$ be any resolution of singularities, i.e., a proper birational morphism from a regular surface $\tilde S$. Then one can  define the Mumford pullback of Weil divisors $f^*: Z_1(S)\to Z_1(\tilde S)_{\QQ}$ (see \cite[Example 7.1.16]{Fu}). This is defined by the property that  for every Weil divisor $D$ on $S$  the difference of $f^*D$ and the proper transform of $D$ is supported on the exceptional locus 
and for any irreducible comonent $C$ of the exceptional locus of $f$ the intersection number
$f^*D.C$ vanishes.

Note that there exists a positive integer $N$ such that the image of $f^*$ is contained in $\frac{1}{N}Z_1(\tilde S)\subset Z_1(\tilde S)_{\QQ}$. We can take as $N$, e.g., the product over all $x\in S$ of the determinants of intersection matrices of irreducible components of $f^{-1}(x)$.

The above pullback descends to the pullback homomorphism  $f^*: A_1(S)\to A_1(\tilde S)_{\QQ}$.
This allows us to define Mumford's intersection product
$$A_1(S)\otimes A_1(S)\to A_0(S)_{\QQ}$$
by the formula 
$$\alpha\cdot \beta = f_*(f^*\alpha \cdot f^*\beta)$$
(see \cite[Example 8.3.11]{Fu}). Here $f^*\alpha \cdot f^*\beta$ denotes the intersection product of the class of the Cartier $\QQ$-divisor $f^*\alpha$ with the class $f^*\beta$ of a $1$-cycle
(see \cite[2.3]{Fu}). 
This product is symmetric (by  \cite[Theorem 2.4]{Fu}), bilinear and independent of the choice of resolution $f$. By definition, the image of Mumford's intersection product is contained in 
$\frac{1}{N} A_0(S)\subset A_0(S)_{\QQ}$ for $N$ as above.

Note that since $\tilde S$ is only regular, we cannot use the intersection product on non-singular varieties as defined in \cite[Chapter 8]{Fu}. However, instead of the above intersection product on $\tilde S$ one can equivalently  use the rational intersection product provided by \cite[Proposition 3.9]{Kl} (see also \cite[Proposition 2.3]{Vi}) or the intersection product provided by Quillen's higher K-theory by using \cite[20.5, Theorem]{Fu}.

In the following we also write $\alpha.\beta $ to denote $ \int_S \alpha\cdot \beta$. 
By abuse of notation we also write $\alpha ^2$ to denote $\alpha.\alpha$.
By construction the image of the corresponding intersection product
$$A_1(S)\otimes A_1(S)\to {\QQ}, \quad (\alpha, \beta)\to \alpha.\beta$$ 
is contained in the discrete subset $\frac{1}{N}\ZZ\subset \QQ$.

\subsection{Ampleness and nefness on normal proper surfaces}

Let $S$ be a normal proper surface over an algebraically closed field $k$. We say that an $\RR$-Weil divisor $H$ on $S$ is nef if for every closed curve $C\subset S$ we have $H.C\ge 0$. We say that an  $\RR$-Weil divisor $H$ is \emph{numerically ample} if for every closed curve $C\subset S$ we have $H.C> 0$ and $H^2>0$. 

The following lemma, showing that such divisors always exist, seems to be well-known but we recall its proof for convenience of the reader. 

\begin{Lemma}\label{existence-num-ample}
	Every proper normal surface $S$ admits a numerically ample Weil divisor.
\end{Lemma}

\begin{proof}
	Let $f:\tilde S\to S$ be a resolution of singularities with projective $\tilde S$ (to obtain such a resolution one can use Chow's lemma and then a resolution of singularities).
Let $A$ be an effective ample Cartier divisor on $\tilde S$. If $C$ is an effective Weil divisor on $S$ then
	$f^*C$ is also effective. Indeed, if we write 
	$f^*C=f_*^{-1}C+\sum a_i E_i$, where $a_i\in \QQ$ and  $E_i$ are the irreducible components of the exceptional locus of $f$, then for all $j$
	$$(\sum a_iE_i).E_j=-f_*^{-1}C.E_j\le 0,$$
	which implies that $a_i\ge 0$ (see, e.g., \cite[1.1]{Gi}).
	This shows that $H=f_*A$ satisfies inequality
	$$H.C=A.f^*C>0.$$
Since by construction $H$ is effective, this also shows that $H^2>0$.
\end{proof}

\medskip

Let $N(S)$ denote the group of Weil divisors on $S$ modulo the radical of Mumford's intersection pairing. One can check that $N(S)$ is isomorphic to the quotient of the group $B_1(S)$ of Weil divisors modulo algebraic equivalence by torsion. If $H$ is a numerically ample $\RR$-Weil divisor on $S$ then passing to a resolution of singularities one can easily see that Mumford's intersection form is negative definite on $H^{\perp}\subset N(S)_{\RR}$, i.e., Hodge's index theorem holds on normal proper (possibly non-projective!) surfaces.
If $H$ is an $\RR$-Cartier divisor then by the Nakai--Moishezon criterion $H$ is numerically ample if and only if it is ample (see \cite[Theorem 1.3]{FM}).  In particular, a numerically ample 
$\RR$-Weil divisor on a non-projective surface is not numerically equivalent to a (numerically) ample $\RR$-Cartier divisor. Note also that by \cite{Sch} there exist normal proper surfaces $S$ with the trivial intersection form on $\Pic S$ (or even with trivial $\Pic S$, i.e., with no non-trivial line bundles).

\subsection{The zeroth and first Chern characters on normal varieties}\label{0-1-Chern-character}

Let $X$ be a normal variety of dimension $n$. The aim of this subsection is to define the Chern characters
$$\ch_i^M: K_0(X)\to A_{n-i}(X)_{\QQ}$$
for $i=0, 1$.

Embedding  $j: X_{reg }\hookrightarrow X$ of the regular locus  induces the restriction  homomorphism $j^*: A_*(X)\to A_*(X_{reg})$.
By normality of $X$ the closed subset $Y=X\backslash X_{reg}\subset X$ has codimension $\ge 2$ and hence the exact sequence
$$A_*(Y)\to A_*(X)\to A_*(X_{reg})\to 0$$ 
shows that $j^*$ is an isomorphism on the Chow groups  of $(n-1)$ and $n$-cycles. So for any class $\alpha$  in the Grothendieck group $  K_0(X)$ of coherent sheaves on $X$, we can define $\ch_0^M(\alpha)$ and $\ch_1^M(\alpha)$ by setting 
$$\ch_i^M(\alpha):=(j^*)^{-1} (\ch_i (j^*\alpha)\cap [X_{reg}])$$
for $i=0,1$. For any 
class $\alpha \in K_0(X)$ we consider $\ch_0^M(\alpha)$ as an integer (using the canonical isomorphism  $ A_n(X)\to\ZZ$). If $\alpha$ is the class of a coherent $\cO_S$-module $E$ then $\ch _0^M (E)$ is equal to the the rank of $E$ at the generic point of $X$.
If $E$ is a reflexive  coherent $\cO_X$-module of rank $1$ then there exists a Weil divisor $D$ with $E\simeq \cO_X(D)$. Then $\ch_1^M(E)=c_1^M(E)$ is the class of $D$ in $A_{n-1}(X)$.
If $E$ is a torsion free coherent $\cO_X$-module of rank $r$ then $\det E:= (\bigwedge ^r E)^{**}$ is a reflexive sheaf of rank $1$ and $c_1^M(E)=c_1^M(\det E)$. In particular, the anticanonical divisor $-K_X=c_1^M(T_X)$ is the unique class of Weil divisors extending $-K_{X_{reg}}$ (note that unlike the cotangent sheaf $\Omega_X$, the tangent sheaf $T_X$ is torsion free and even reflexive).

 \section{Chern character on normal surfaces}
 
Let $S$ be a normal surface. 
By Subsection \ref{0-1-Chern-character} to define a homomorphism
$$ \ch^M=(\ch_0^M, \ch_1^M, \ch_2^M): K_0(S)\to A_*(S)_{\QQ}$$
we need only to define $\ch_2^M:K_0(S)\to A_2(S)_{\QQ}$. For $\alpha\in K_0(S)$ we define it by the formula
$$\ch_2^M(\alpha):=
\tau _S(\alpha)-[S]- \left( \ch_1^M(\alpha ) -\frac{1}{2}K_{S} \right) +
\frac{1}{2}\ch_1^M(\alpha)\cdot   K_{S}-\ch _0^M (\alpha)\, \Td _2 ({S}).$$
Let us recall that by definition $\Td (S)=\tau _S  (\cO_S)$. If $S$ is smooth then $\Td (S)=\td (S)\cap [S]$.

The homomorphism $ \ch^M: K_0(S)\to A_*(S)_{\QQ}$ is called \emph{Mumford's Chern character}.
Let $K^0(S)$ be the Grothendieck group of vector bundles on $S$.
Composing $\ch^M$ with the canonical homomorphism $\varphi: K^0(S)\to K_0(S)$ one immediately sees that for any $\alpha \in K^0(S)$ 
$$\ch^M(\varphi (\alpha))= \ch(\alpha)\cap [S],$$
where $\ch: K^0(S)\to A^*(S)_{\QQ}$ is the usual Chern character of vector bundles. In particular, if $S$ is smooth then $\varphi$ is an isomorphism and $\ch^M (\alpha)= \ch(\alpha)\cap [S]$ for all $\alpha \in K_0(S)$. 

\medskip

By the covariance property for proper morphisms (see \cite[Theorem 18.3, (1)]{Fu}), if $S$ is a proper normal surface then  $\chi(S, E)=\int_S\tau _S(E).$ In this case the above definition implies that for any coherent $\cO_S$-module $E$ we have the Riemann--Roch formula
	\begin{equation}
		\chi (S, E)=\int_S \ch_2^M(E)-\frac{1}{2}\ch_1^M(E).K_S+\ch_0^M(E)\, \chi (S, \cO_S) .\label{RR}
	\end{equation}

\subsection{Computation of $\ch^M$}\label{computation-ch^M}
 
In this subsection we show how to compute $\ch^M$ using a resolution of singularities.
Let $\Refl (S)$ denotes the category of reflexive coherent $\cO_S$-modules.
The following lemma says that it is sufficient to compute $\ch^M$ on $\Refl (S)$:

\begin{Lemma}\label{reduction-to-reflexive}
	Let $S$ be a normal surface. Then the canonical homomorphism $K(\Refl (S))\to K_0(S)$, from the Grothendieck group of $\Refl (S)$ to $K_0(S)$, is  
	an isomorphism.
\end{Lemma}

\begin{proof}
	The proof is analogous to the one showing that on a smooth variety $X$ the canonical homomorphism
	$K^0(X)\to K_0(X)$ is an isomorphism. So it is sufficient to show that
	any coherent $\cO_S$-module $F$ has a finite resolution 
	$$0\to E_2\to E_1\to E_0\to F\to 0$$
	in which all $E_i$ are reflexive coherent $\cO_S$-modules. By \cite[Theorem 2.1]{SV} $S$ has sufficiently many vector bundles and hence there exists an exact sequence of the form
	$$ E_1\mathop{\longrightarrow}^{\varphi} E_0\to F\to 0$$
	in which both $E_1$ and $E_0$ are vector bundles (in particular, they are reflexive). Then
	$\im \varphi$ is torsion free (as it is contained in $E_0$) and hence $\ker \varphi$ is reflexive by \cite[Lemma 31.12.7]{SP}.
\end{proof}

\medskip
 
 Let $f: \tilde S\to S$ be any resolution of singularities of $S$ (existence of such resolutions has been established by J. Lipman; see \cite[Theorem 54.14.5]{SP}).
 By \cite[Theorem 18.3, (1)]{Fu} for any coherent $\cO _{\tilde S}$-module $F$ we have
 $$\tau _S (Rf_*F)= f_* (\tau_{\tilde S} (F)).$$
 In particular, if $F$ is a vector bundle on $\tilde S$ and $E =(f_{* }F)^{**}$ then we have
 $$\tau_{S}( E)=f_* (\tau_{\tilde S} (F))+\tau_{S}(R^1f_*F)+ \tau _S (E/f_*F).$$
 For a sheaf $G$ on $S$ supported in dimension $0$ and a closed point $x\in X$
 we denote by $l_x(G)$ the length of the stalk $G_x$ over the local ring $\cO_{X,x}$. 
 Since both $R^1f_*F$ and $E/f_*F$ are supported on a finite set of points (the points over which $f$ is not an isomorphism), \cite[Example 18.3.11]{Fu} implies that
 $$\tau_{S}(R^1f_*F)=\sum _{x\in S} l_x(R^1f_*F)\, [x]$$
 and 
 $$\tau_{S}(E/f_*F)=\sum _{x\in S} l_x(E/f_*F)\, [x].$$
 Therefore 
 $$\tau_{S}( E)=f_* (\tau_{\tilde S} (F))+\sum _{x\in S} \chi (x, F)\, [x],$$
 where $\chi (x, F):=l_x(E/f_*F)+ l_x(R^1f_*F).$
 In particular, 
$$\Td (S)=\tau _S  (\cO_S)= f_* (\Td({\tilde S}))+\sum _{x\in S} \chi (x, \cO_{\tilde S})\, [x]$$
recovering the formula from  \cite[Example 18.3.4]{Fu}.

For every point $x\in S$ there exists a unique Weil $\QQ$-divisor $c_1(x, F)$ supported on the set $f^{-1}(x)$ such that for every irreducible component $C$ of $f^{-1}(x)$ we have
$$c_1(x, F).C=\int_C c_1(F)\cap [C].$$
 Note that $f_*(\ch_1^M(F))=\ch_1 ^M(E)$ and
 $$\ch_1^M(F)=f^*\ch_1^M(E) + c_1(f, F).$$
where $c_1(f, F):= \sum _{x\in S} c_1(x, F)$.
 By definition we have
 $$\ch (F)\cap \td ({\tilde S})= 1+\left( c_1(F) +\frac{1}{2}c_1({\tilde S}) \right) +
 \ch_2(F)+\frac{1}{2}\ch_1(F)\cap  c_1({\tilde{S}})+\ch _0(F)\cap\td _2 (\tilde{S}).$$  
 Writing $K_f= -c_1(f, T_{\tilde S})$, we have
$$\ch_1(F)\cap [K_{\tilde{S}}] = (f^*\ch_1^M(E) +c_1(f, F))\cdot (f^*K_S+K_f)=f^* (\ch_1^M(E)\cdot K_S)+c_1(f, F)\cdot K_f.$$ 
 It follows that
 \begin{align*}
 	f_* (\tau_{\tilde S} (F))= f_* (\ch (F)\cap \Td ({\tilde S}) )=
 [S]+ \left( c_1^M(E) -\frac{1}{2}K_{S} \right) +f_*(\ch_2(F)\cap [{\tilde S}])\\
 -\frac{1}{2}\left( \ch_1^M(E)\cdot K_S+f_*(c_1(f, F)\cdot K_f)
 \right)
 +\ch ^M_0(E)\, \left(\Td_2 (S)-\sum _{x\in S} \chi (x,\cO_{\tilde S})\, [x]\right).
 \end{align*}
 Using the definition of $\ch_{2}^M$ and the above formula, we get
 \begin{align*}
 	\ch_2^M (E)&=\left(f_* (\tau_{\tilde S} (F))+\sum _{x\in S} \chi (x, F)\, [x])\right)-[S]- \left( c_1^M(E) -\frac{1}{2}K_{S} \right) +
 	\frac{1}{2}\ch_1^M(E)\cdot   K_{S}-\ch _0^M (E)\, \Td _2 ({S})\\ 
 &=	f_*(\ch_2(F)\cap [{\tilde S}])+\sum _{x\in S}\left(\chi (x, F)-\chi (x, \cO_{\tilde S})\, \ch_0 (F)
 	-\frac{1}{2}c_1(x, F)\cdot K_{\tilde S}\right)[x].
 \end{align*} 
 In particular, the above formula for $\ch_2^M (E)$ does not depend on the choice of $F$ such that $E=(f_*F)^{**}$.

To finish the computation it is sufficient to note that for every $E\in \Refl (S)$ there exists 
a vector bundle $F$ on $\tilde S$ such that  $E=(f_*F)^{**}$ (e.g., we can take $F=(f^*E)^{**}$).

\medskip

\begin{Remark}
Note that 
$$f_*(\ch_1^M(F))^2=(\ch_1^M(E))^2 + \sum _x c_1(x, F)^2[x].$$
So if we set $c_2^M(E):=\frac{1}{2}(c_1^M(E))^2-\ch_2^M (E)$ then 
\begin{align*}
c_2^M(E)=f_*(c_2(F)\cap [{\tilde S}])-\sum _{x\in S}
\left(\chi (x, F)-\chi (x, \cO_{\tilde S})\, \ch_0 (F)
+\frac{1}{2}c_1(x, F)\cdot (c_1(x, F)-K_{\tilde S}) \right)
[x].
\end{align*} 
\end{Remark}

\begin{Remark}
 	In  \cite{La-MZ} and \cite{La-Inters} the author introduced and studied a completely different notion of the Chern character on normal surfaces, that is defined only for reflexive sheaves on normal surfaces over an algebraically closed field. The relation between this character and the one considered here is provided by the formula
 		$$c_2^M(E)= c_2(E)-\sum _{x\in \Sing S}a(x, E)[x]$$
 	that holds in $A_0(S)_{\RR}$ (for proper normal surfaces this follows from \cite[Theorem 4.4]{La-Inters} and a proof in the non-proper case is similar). 
\end{Remark}

 \subsection{Operational Chern character on normal surfaces}

 It is well-known that if a normal surface $S$ over a perfect field has a resolution of singularities, whose exceptional divisor consists of only rational curves, then $S$ is an Alexander scheme, i.e.,  for every $f: Y\to S$ the evaluation map $ev_X: A^*(Y\to S)\to A_*(Y)$, $\alpha\to \alpha \cap [S]$, is an isomorphism (see \cite[Theorem 4.1]{Vi}). However, this is no longer true for arbitrary normal surfaces (see \cite[Example 4.8]{Ki}). More precisely, there exist normal proper surfaces for which the evaluation map $A^*(S)\to A_*(S)$ is not an isomorphism.
 
 Let $B_* (S)$ be the quotient of $A_*(S)$ by algebraic equivalence 
 and let $B^*(S)=B^*(S\to S)$ be the corresponding operational Chow cohomology. Then our construction of Mumford's Chern character together with \cite[Theorem 4.1]{Vi} and \cite[Remark 4.5]{Ki}, imply the following corollary:

 \begin{Corollary}
 	Let $S$ be a normal surface over a perfect field. Then there exists a well-defined homomorphism  $ \ch^M: K_0(S)\to B^*(S)_{\QQ}$
 	such that for any $\alpha \in K_0(S)$ we have
 	$$\tau _S (\alpha)= \ch^M(\alpha)\cap \Td (S)$$
 	in $B_*(S)_{\QQ}$.
 	Moreover, if $S$ admits a resolution of singularities  such that the irreducible components of the exceptional divisor are rational curves then we can lift the above homomorphism to a homomorphism
 	$ \ch^M: K_0(S)\to A^*(S)_{\QQ}$ such that 
 	$$\tau _S (\alpha)= \ch^M(\alpha)\cap \Td (S) $$
 	in $A_*(S)_{\QQ}$.
 \end{Corollary}
 
 \begin{Remark}
 	The author of \cite{Ki} claims that $B^*(S)_{\QQ}\to B_*(S)_{\QQ}$ is an isomorphism for any normal surface $S$ over any field. This would imply that we do not need assume in the above corollary that the base field is perfect. However, the paper \cite{Ki} mixes the notion of regularity with smoothness. This causes problems in various places of this paper. The author of \cite{Vi} is more careful and one can follow his proof of  \cite[Theorem 4.1]{Vi}  to reduce to the algebraically closed case when regularity and smoothness coincide.
 \end{Remark} 
 
 \section{Proof of Theorem \ref{main}}
 
 Let $f: \tilde S\to S$ be the minimal resolution of singularities of $S$. 
 By the Leray spectral sequence, for any coherent $\cO _{\tilde S}$-module $F$ we have
 $$\chi (S, Rf_*F)= \chi (\tilde S, F).$$
 In particular, if $E$ is a coherent reflexive $\cO_S$-module and $F$ is the reflexivization $(f^{* }E)^{**}$ of $f^*E$ then $f_*F=E$ and 
 $$\chi (S, E)= \chi (\tilde S, F)+h^0(S,R^1f_*F). $$
 Let $r$ be the rank of $E$ and let us recall that we set $c_1(f, F)= \ch_1^M(F)-f^*\ch_1^M(E)$. Then by the definition of Mumford's discriminant and the Riemann--Roch formula (see Section 2,  equation (\ref{RR})) we have
 \begin{align*} 
 	\int_{\tilde S} \Delta^M(F)-\int_{S} \Delta^M(E)=&2r(\chi (S, E)-\chi (\tilde S, F))+ 2r^2 (\chi (\tilde S, \cO_{\tilde X})-\chi (S, \cO_S))\\
 	&+(( \ch_1^M(F))^2-( \ch_1^M(E))^2)
 	+r( \ch_1^M(E).K_S- \ch_1^M(F).K_{\tilde S})\\
 	=&2rh^0(S,R^1f_*F)-2r^2h^0(S,R^1f_*\cO_{\tilde S})
 	+c_1(f, F)^2-r c_1(f, F).K_{\tilde S}.\\
 \end{align*}  
 Alternatively, the above formula follows from the more general computation in Subsection \ref{computation-ch^M}.
 Now \cite[Proposition 3.9 and Corollary 3.10]{La-Inters} imply that there exists a constant $D_S$ depending only on the local type of singularities of $S$ such that for every rank $r$ reflexive $\cO_S$-module $E$ we have
 $$\left|\int_S \Delta^M(E)-\int_{\tilde S} \Delta^M(F)\right| \le D_Sr^2.$$
If $S$ has only rational Gorenstein singularities then 
$$\int_{\tilde S} \Delta^M(F)=\int_{ S} \Delta^M(E)+c_1(f, F)^2\le \int_{S} \Delta^M(E).$$
Indeed, since $R^1f_*\cO_{\tilde S}=0$ we have $R^1f_*F=0$ (this follows, e.g., from \cite[Proposition 3.9]{La-Inters}). Vanishing of $c_1(f, F).K_{\tilde S}$ in this case follows from the equality $K_{\tilde S}=f^*K_S$.

 Together with usual Bogomolov's inequality in characteristic zero and \cite[Theorem 1.1]{Ko}  in positive characteristic, the above inequalities imply Theorem \ref{main} for reflexive sheaves. In general, if $E$ is a torsion free coherent $\cO_S$-module then
 $\chi (S, E^{**})\ge \chi (S, E)$ and $\ch_i (E)= \ch_i(E^{**})$ for $i=0,1$. This implies 
 $$\int_S \Delta^M(E)\ge \int_S \Delta^M(E^{**}),$$
 which proves  Theorem \ref{main}.
 
 \medskip
 
 \begin{Remark}
 The inequality
 $$\int_{\tilde S} \Delta^M(F)\le \int_{ S} \Delta^M(E)$$
 for a reflexive sheaf $E$ on a normal proper surface $S$ with at most rational Gorenstein singularities was first observed in
 \cite[(4.6)]{Th} following  \cite{La-MZ}. The above proof is essentially equivalent but simpler.
 
 The above inequality was also claimed in \cite[Proposition 3.7]{NS} but the proof depends on \cite[Lemma 3.5]{NS}, whose proof contains various mistakes. 
 \end{Remark}

 \begin{Remark}
 We expect that one can obtain a better inequality than the one in Theorem \ref{main} by replacing 
 Mumford's discriminant $\Delta^M(E)$ with the discriminant $\Delta(E)$ defined in \cite{La-Inters}. Unfortunately, the proof of this more precise inequality depends on \cite[Conjecture 8.2]{La-MZ}.
 \end{Remark}

\section{Bridgeland stability conditions on normal surfaces}

Let $S$ be a normal proper surface over an algebraically closed field. Let us fix arbitrary $\RR$-Weil divisors $B$ and $H$.
We define a group homomorphism $Z_{H,B}: K_0(S)\to \CC$, called the \emph{charge}, by 
$$Z_{H,B}(E):=-\int _S e^{-(B+iH)}\ch ^M(E)+\frac{C_{S} }{2}\ch_{0}^M(E).$$
Here $C_S$ is the constant from Theorem \ref{main} and
$\int _S e^{-(B+iH)}\ch ^M(E)$ is defined as the degree of the class 
$$e^{-(B+iH)}\ch ^M(E):=\ch _2^M(E)-\ch_1^M(E)\cdot (B+iH)+\frac{1}{2}\ch_0 ^M(E) \, (B+iH)^2 \in A_*(S)_{\CC}.$$
Note that for if $B$ and $H$ are not Cartier then this is not the same as the degree of 
$\ch ^M(E(-B-iH))$.

If $H$ is a numerically non-trivial nef $\RR$-Weil divisor then we can define the \emph{$H$-slope} of a (non-zero) torsion free coherent $\cO_S$-module $E$ by 
$$\mu_H(E)=\frac{\ch_1^M(E).H}{\ch_0^M(E)}.$$
This allows us to consider slope $H$-(semi)stability of torsion free sheaves.
Every coherent $\cO_S$-module $E$ admits a unique Harder--Narasimhan filtration
$$E_0\subset E_1\subset ...\subset E_m=E$$
characterized by the property that $E_0$ is torsion, all quotients $E_i/E_{i-1}$, $i=1,...,m$ are torsion free slope $H$-semistable and  
$$\mu_{\min, H}(E)=\mu (E_m/E_{m-1})< ...<\mu (E_1/E_{0})=\mu_{\max, H}(E). $$
To simplify notation, in case $E$ is torsion we set  $\mu_{\min, H}(E)=+\infty$.
We define a torsion pair $(\cT_{H,B}, \cF_{H,B}) $ in the category $\Coh (S)$ of coherent $\cO_S$-modules by setting
$$\cT_{H,B}:=\{ E\in \Coh (S):  \mu_{\min, H}(E)> B.H\}$$
and 
$$\cF_{H,B}:=\{ E\in \Coh (S):  \hbox{ $E$ is torsion free and } \mu_{\max, H}(E)\le B.H\} .$$
We also consider the tilted abelian category
$$\Coh ^{H, B}(S):= \langle \cF_{H,B}[1], \cT_{H,B}\rangle \subset D^b(S).$$

The proof of the following proposition follows that of  \cite[Lemma 6.2]{Br} and \cite[Corollary 2.1]{AB}.

\begin{Proposition}
If $H$ is numerically ample then the pair $(Z_{H,B}, \Coh ^{H, B}(S))$ is a Bridgeland slope function.
\end{Proposition}

\begin{proof}
For any $E\in \Coh (S)$ of positive rank we have 
\begin{align*} 
	\Re Z_{H,B}(E)&=-\int _S \ch _{2}^M(E) + \ch_1(E).B-\frac{\ch_0(E)}{2}(B^2-H^2)
	+\frac{\ch_{0}(E) }{2}C_{S}\\
	&=\frac{1}{2\ch_0(E)} \int _S \Delta ^{M}(E)- \frac{1}{2\ch_0(E)} \left( \ch_1^B(E)\right)^2+\frac{\ch_{0}(E)}{2}(H^2+{C_{S} }),
\end{align*} 
where $\ch_1^B(E)=\ch_1(E)-\ch_0(E)\, B$.

The proof of \cite[Corollary 2.1]{AB} goes through if we can show that for any such $E$
with $\mu_H(E)=B.H$ we have $\Re Z_{H,B}(E)>0$.
If $\mu_H(E)=B.H$ then $H. \ch_1^B(E)=0$, so by the Hodge index theorem $ \left( \ch_1^B(E)\right)^2\le 0$ (here we use only that $H$ is nef and $H^2>0$).
So Theorem \ref{main} implies that
$$\Re Z_{H,B}(E)\ge \frac{1}{2\ch_0(E)} \left(\int _S \Delta ^{M}(E) +C_{S}\, \ch_{0}(E)^2 \right)+ \frac{\ch_{0}(E)}{2}H^2 \ge \frac{\ch_{0}(E)}{2}H^2 >0.$$
\end{proof}

\begin{Remark}
	The fact that Theorem \ref{main} implies the above proposition is standard and contained in \cite{Ko}. Unfortunately, the proof of \cite[Theorem 6.6]{Ko} contains small computational errors. We sketched proof of the above proposition also to convince the reader that we do not need to use ampleness. In fact, the above part of the proof uses only the fact that $H$ is nef and big. The full strength of numerical ampleness is used in the omitted part of the proof, when dealing with torsion sheaves.
\end{Remark}

\medskip

Now we need to define a finite rank lattice $\Lambda$ and a surjective group homomorphism 
$v: K_0(S)\twoheadrightarrow \Lambda $ such that $Z_{H,B}$ factors through $v$.
To do so let us note that  $Z_{H,B}$ factors through 
$$K_0(S)\mathop{\longrightarrow}^{\ch^M}A_*(S)\to B_*(S).$$
We have canonical isomorphisms $\int_S: B_0(S)\to \ZZ$ and $B_2(S)\to \ZZ$ and by the generalized version of the theorem of the base (see \cite[Th\'eor\`eme 3]{Kah}),
$B_1(S)$ is a finitely generated $\ZZ$-module. So as $\Lambda$ we can take the image of 
the homomorphism $K_0(S)\to B_*(S)\to B_*(S)/\hbox{Torsion}\simeq \ZZ \oplus N(S)\oplus \ZZ$. 

In the following, the $\RR$-linear map $V=\Lambda\otimes \RR \to \CC$ induced by $Z_{H,B}: K_0(S)\to \CC$ is by abuse of notation also denoted by $Z_{H,B}$. 
It is possible to take $C_S$ satisfying Theorem \ref{main} so large that for every effective Weil divisor $D$ on $S$ we have
$$C_S (H.D)^2 +D^2\ge 0$$
(see \cite[Exercise 6.11]{MS} or \cite[Lemma 6.5]{Ko}). We fix such $C_S$ and
 we  consider the quadratic form $Q:V\to \RR$ defined by
$$ Q(\alpha):= \int _S\Delta^M(\alpha)+C_S(\ch_1^B (\alpha).H)^2$$
for $\alpha \in V$. This form is used to check the support property in 
 the following theorem:
 
\begin{Theorem}\label{explicit-stability-on-proper}
Let $S$ be a normal proper surface. Let $H$ be a numerically ample $\RR$-Weil divisor and $B$ an arbitrary $\RR$-Weil divisor. Then the pair $(Z_{H,B}, \Coh ^{H, B}(S))$ defines a geometric Bridgeland stability condition on $D^b(S)$. This stability condition satisfies the support property condition with respect to $(\Lambda, v)$.
\end{Theorem}

We skip proof of this theorem as it is the same as in the smooth case and one can find a full account of the proof in \cite[Sections 5 and 6]{MS}.

\section*{Acknowledgements}

The author would like to thank E. Macri and S. Schr\"oer for some remarks. The author would also like to thank the referee for careful reading of the paper and the remarks.
 
The paper was written while the author was an External Senior Fellow at Freiburg Institute for Advanced Studies (FRIAS), University of Freiburg, Germany. The author would like to thank Stefan Kebekus for his hospitality during the author's stay in FRIAS.

The  author was partially supported by Polish National Centre (NCN) contract number 2021/41/B/ST1/03741. The research leading to these results has received funding from the European Union's Horizon 2020 research and innovation programme under the Maria Sk{\l}o\-do\-wska-Curie grant agreement No 754340.

\end{document}